\documentclass[a4paper,english,sumlimits,reqno,12pt]{amsart}

\usepackage[latin1]{inputenc}
\usepackage[T1]{fontenc}
\usepackage{babel}

\usepackage{color}
\usepackage{graphicx}
\usepackage{bbm}

\usepackage{fullpage}

\usepackage{pstricks}
\usepackage{pstricks-add}
\usepackage{pst-node}

\usepackage{fp}

\usepackage[section]{placeins}

\usepackage{enumerate}

\usepackage{varioref}

\usepackage{amssymb}
\usepackage{amsmath}
\usepackage{amsthm}

\usepackage[mathcal]{eucal}
\usepackage{mathrsfs}

\newcounter{proof}
\newenvironment{myproof}%
{\stepcounter{proof}\begin{proof}}%
{\end{proof}}%
\newcounter{proofstep}[proof]
\newenvironment{proofstep}%
{\refstepcounter{proofstep}\bigskip\par\noindent%
  \textbf{Step \theproofstep.}%
  \noindent}%
{\par}%

\theoremstyle{plain}

\newtheorem{thm}{Theorem}[section]
\newtheorem*{thm*}{Theorem}

\newtheorem{lem}[thm]{Lemma}
\theoremstyle{definition}

\theoremstyle{remark}

\numberwithin{equation}{section}

\newcommand{\bb}[1]{\ensuremath{\mathbb #1}}
\newcommand{\cal}[1]{\ensuremath{\mathcal #1}}
\newcommand{\scr}[1]{\ensuremath{\mathscr #1}}

\newcommand{\bfpar}[1]{\paragraph{\textbf{#1}}}

\newcommand{\union}{\cup}
\newcommand{\Union}{\bigcup}
\newcommand{\isect}{\cap}

\DeclareMathOperator{\umd}{UMD}

\newcommand{\charfun}{\scalebox{1.0}{\ensuremath{\mathbbm 1}}}
\newcommand{\dcubes}{\scr Q}

\newcommand{\dif}{\ensuremath{\, \mathrm d}}

\DeclareMathOperator{\supp}{supp}

\DeclareMathOperator{\lin}{span}

\DeclareMathOperator{\cond}{\bb E}

\DeclareMathOperator{\salg}{\text{$\sigma$-algebra}}
\DeclareMathOperator{\pred}{\pi}
\DeclareMathOperator{\lev}{lev}
\DeclareMathOperator{\proj}{pr}
\newcommand{\pow}{\ensuremath{\scr P}}
\newcommand{\stripe}{\ensuremath{\scr S}}

\renewcommand{\mod}{\operatorname*{mod}}
\renewcommand{\mid}{\operatorname*{|}}

\begin{document}
\title{Adaptive deterministic dyadic grids on spaces of homogeneous type}

\date{\today}
\author[R. Lechner]{Richard Lechner}
\address{
  Richard Lechner,
  Institute of Analysis,
  Johannes Kepler University Linz,
  Altenberger Strasse 69,
  A-4040 Linz, Austria
}
\email{Richard.Lechner@jku.at}
\author[M. Passenbrunner]{Markus Passenbrunner}
\address{
  Markus Passenbrunner,
   Institute of Analysis,
   Johannes Kepler University Linz,
   Altenberger Strasse 69,
   A-4040 Linz, Austria
}
\email{Markus.Passenbrunner@jku.at}

\date{\today}


\begin{abstract}
  In the context of spaces of homogeneous type, we develop a method to deterministically construct
  dyadic grids, specifically adapted to a given combinatorial situation.
  This method is used to estimate vector--valued operators rearranging martingale difference
  sequences such as the Haar system.
\end{abstract}

\maketitle

\section{Introduction}\label{s:intro}

In~\cite{Figiel1988,Figiel1990}, T. Figiel developed martingale methods to prove a vector--valued
$T(1)$ theorem by decomposing the singular integral operator $T$ into an absolutely converging
series of basic building blocks $T_m$ and $U_m$, $m \in \bb Z$.
Those operators are given by the linear extension of
\begin{equation}\label{eq:intro:shift-1}
  T_m h_I = h_{I + m |I|}
  \qquad\text{and}\qquad
  U_m h_I = \charfun_{I + m |I|} - \charfun_I,
\end{equation}
where $\{ h_I \}$ denotes the Haar system on standard dyadic intervals $I$ and $\charfun_A$ the
characteristic function of the set $A$.
The crucial norm estimates obtained in~\cite{Figiel1988} take the form
\begin{align}
  \| T_m\, :\, L_E^p \rightarrow L_E^p \|
  & \leq C\, \big( \log_2 (2 + |m|) \big)^\alpha,
  \label{eq:intro:shift-estimate-1}\\
  \| U_m\, :\, L_E^p \rightarrow L_E^p \|
  & \leq C\, \big( \log_2 (2 + |m|) \big)^\beta,
  \label{eq:intro:shift-estimate-2}
\end{align}
where $1 < p < \infty$ and the constant $C > 0$ depends only on $p$, the $\umd$--constant of the
Banach space $E$ and $\alpha, \beta < 1$.
Estimates~\eqref{eq:intro:shift-estimate-1} and~\eqref{eq:intro:shift-estimate-2} are obtained by
hard combinatorial arguments, analyzing structure and position of dyadic intervals.

T. Figiel's decomposition method was extended in~\cite{MP2012} to spaces of homogeneous type to
obtain a vector--valued $T(1)$ theorem, requiring norm estimates for the building blocks $T_m$ and
$U_m$ in the setting of spaces of homogeneous type.
These estimates are proved by hard combinatorial arguments similar to~\cite{Figiel1988}.

In~\cite{lechner:thirdshift}, an alternative proof for the estimates of $T_m$ and $U_m$ is given which
eliminates the hard combinatorics in~\cite{Figiel1988} to a great extent.
Adapting the dyadic grid by means of an algebraic shift, $T_m$ and $U_m$ are decomposed into roughly
$\log_2(2 + |m|)$ martingale transform operators, thereby yielding~\eqref{eq:intro:shift-estimate-1}
and~\eqref{eq:intro:shift-estimate-2}.
The shift of the dyadic grid is accomplished by the one--third--trick, which originates in the work
of~\cite{davis:1980}, \cite{wolff:1982}, \cite{garnett_jones:1982},
and~\cite{chang_wilson_wolff:1985}.

Adaptive dyadic grids also proved to be a valuable tool for estimating so called stripe operators
in~\cite{lechner:int:2011}.
Those stripe operators were used in~\cite{LeeMuellerMueller2011} as well as
in~\cite{lechner:int:2011} to show weak lower semi-continuity of functionals with separately
convex integrands on scalar--valued $L^p$ and vector--valued $L^p$, respectively.
For the scalar--valued $L^2$ version of this result, cf.~\cite{s_mueller:1999}.

In this paper we extend the method from~\cite{lechner:thirdshift} to construct adapted dyadic grids in
spaces of homogeneous type, which allow us to
\begin{enumerate}[(i)]
\item simplify the combinatorial arguments for the estimation of the rearrangement operators $T_m$
  used in the proof of the $T(1)$ theorem in~\cite{MP2012},
\item generalize the vector--valued result in~\cite{lechner:int:2011} on stripe operators to
  spaces of homogeneous type.
\end{enumerate}

\subsection*{Related recent developments}

In~\cite{hytonen:2011}, T. P. Hyt\"onen presented a proof of T. Figiel's vector-valued $T(1)$
theorem, cf.~\cite{Figiel1990}, based upon randomized dyadic grids, originating
in~\cite{nazarov_treil_volberg:1997,nazarov_treil_volberg:2003}.
By contrast, the method developed in the present paper allows us to adapt a dyadic grid
\emph{deterministically} to a given combinatorial situation.

\section{Preliminaries}\label{s:preliminaries}

In this section we present some basic facts concerning spaces of homogeneous type.
For basic facts on $\umd$--spaces used within this work, the notion of Rademacher type and
cotype as well as Kahane's contraction principle and Bourgain's version of Stein's martingale
inequality, we refer to~\cite{lechner:thirdshift}.

Let $X$ be a set. A mapping $d:X\times X\rightarrow\mathbb{R}_{0}^{+}$ with
the properties that for all $x,y,z\in X$,
\begin{enumerate}
\item $d(x,y)=0\Leftrightarrow x=y$, 
\item $d(x,y)=d(y,x)$, 
\item $d(x,y)\leq K_X\, (d(x,z)+d(z,y))$ for some constant $K_X\geq 1$ only depending on $X$, 
\end{enumerate}
is called a \emph{quasimetric }and $(X,d)$ is called a \emph{quasimetric space.}
Given a quasimetric $d$, we define the ball centered at $x\in X$ with radius $r>0$ as
\begin{equation*}
B(x,r):=\{y\in X:d(x,y)<r\}.
\end{equation*}
Additionally, a set $A\subset X$ is called \emph{open}, if for all $x\in X$ there exists $r>0$ such
that $B(x,r)\subseteq A$.
Furthermore, for an arbitrary set $A\subset X$ and $r>0$, define
\begin{equation*}
  B(A,r):=\{y\in X: d(A,y)<r\}.
\end{equation*}

Let $(X,d)$ be a quasimetric space such that every ball in the quasimetric $d$ is open and $| \cdot
|$ be a Borel measure. If the \emph{doubling condition} holds, i.e. there is a constant $C_d>0$ such
that
\begin{equation}\label{eq:doubling}
 0<| B(x,2r) |\leq C_d | B(x,r) |<\infty,\quad x\in X,\ r>0,
\end{equation}
then $(X,d,| \cdot |)$ is called a \emph{space of homogeneous type.}
Since for a given quasimetric space $(X,d)$, the balls in $X$ are not necessarily open, we added
this condition to the definition. This is the case, if for instance one imposes a Hölder condition on
$d$: There exists $c<\infty$ and $0<\beta <1$ such that for all $x,y,z\in X$ we have
\begin{equation}
|d(x,z)-d(y,z)|\leq c\cdot d(x,y)^\beta \operatorname{max}\{d(x,z),d(y,z)\}^{1-\beta}.\label{eq:lipcond}
\end{equation}
In fact, R. A. Mac{\'{\i}}as and C. Segovia proved in \cite{MaciasSegovia1979} that for every space
of homogeneous type there exists an equivalent quasimetric with the desired Hölder property.
Here, a quasimetric $d'$ is equivalent to a quasimetric $d$ if there exists a finite constant $c$
such that
\begin{equation*}
\frac{1}{c}\, d(x,y)\leq d'(x,y)\leq c\, d(x,y),\qquad x,y\in X.
\end{equation*}

Let $\scr C$ be a collection of arbitrary sets. $\scr C$ is called \emph{nested}, if  $A\cap B\in
\{A,B,\emptyset\}$ for all $A,B\in \scr C$.
Furthermore, for such a given nested collection $\scr C$ we define the predecessor
$\pi_{\scr C}(C)$ of $C$ with respect to the collection $\scr C$ by
\[
\pi_{\scr C}(C):=\bigcap\big\{D:D\supsetneq C,D\in\scr C\cup\{X\}\big\}.
\]

\bfpar{Dyadic cubes}
In spaces of homogeneous type, one can construct a collection of subsets that has similar properties
to dyadic cubes in $\bb R^k$, cf. M. Christ~\cite{Christ1990} and G. David~\cite{David1991}.

\begin{thm}\label{th:dyadic}
  Let $(X,d,| \cdot |)$ be a space of homogeneous type. Then there exists a system
  $\scr Q$ of open subsets of $X$ with centers $m_A\in A$ of $A\in\scr Q$ and a
  splitting $\scr Q=\cup_{n\in\mathbb Z}\scr Q_n$ such that the following
  properties are satisfied with uniform constants
  $q<1,$ $C_{1} ,C_{2},C_{3},\eta\in\mathbb{R}^{+},N\in \mathbb{N}$: 

  \begin{enumerate}
  \item For all $n\in\mathbb{Z}$ we have that $X=\bigcup_{A\in\scr Q_n}A$ up to $|
    \cdot |$-null sets.

  \item For $A\in\scr Q_k$ and $B\in\scr Q_n$ with $k\leq n$, we have either
    $B\subset A$ or $A\cap B=\emptyset.$

  \item For each $B\in \scr Q_n$ and every $k\leq n$, there is exactly one
    $A\in\scr Q_k$ such that $B\subset A$.

  \item For all $n\in\mathbb{Z}$ and $A\in\scr Q_n$ we have that
    $B(m_A,C_{1}q^{n})\subseteq A\subseteq B(m_A,C_{2} q^{n})$.
    \label{th:dyadic-4}

  \item Let $A\in\scr Q_n$.
    The boundary layer of $A$ having width $t$ is given by
    \begin{equation}\label{eq:definitionstrip}
      \partial_{t}A:=\{x\in A:d(x,X\setminus A)\leq tq^{n}\},
    \end{equation}
    and satisfies the measure estimate
    \begin{equation}\label{eq:estimatestrip}
      |\partial_{t}A |<C_{3}t^{\eta}|A|.
    \end{equation}\vspace{-0.4cm}

  \item For all $n\in\mathbb{Z,}$ the collection $\scr Q_n$ is countable.

  \item For all $n\in\mathbb{Z}$ and $A\in \scr Q_n$ we have
    $N(A):=|\{B\in \scr Q_{n+1}:B\subseteq A\}|\leq N$.

  \item For all $n\in\mathbb{Z}$ and $A\in\scr Q_n$ there exists a subcollection $\scr S$ of
    $\scr Q_{n+1}$ with $|\scr S|\leq N$ and
    \[
    A=\bigcup_{B\in\scr S} B\quad \text{up to }|\cdot|\text{-null sets}.
    \]
  \end{enumerate}
\end{thm}
\noindent
We define the \emph{level} of a cube $A\in\dcubes_{n}$ as $\lev A:=n$, and furthermore
\begin{equation}\label{eq:scaling}
  r\diamond A:=B(A,r q^{\lev A}),
  \qquad A\in \dcubes,\ r > 0.
\end{equation}

In the following, $(X,d,|\cdot|)$ denotes a space of homogeneous type, equipped with a
quasimetric $d$ and a measure $|\cdot|$.

\begin{lem}\label{lem:diamond}
  Let $A \in \dcubes$ and $r > 0$.  Then
  \begin{equation*}
    r \diamond A
    \subset B\big(m_A,\, K_X (C_2 + r)q^{\lev A}\big).
  \end{equation*}
\end{lem}

\begin{proof}
  Let $z \in r\diamond A = B(A,\, rq^{\lev A})$ and estimate
  \begin{align*}
    d(m_A,\, z)
    & \leq \inf_{y\in A} K_X \big( d(m_A,y) + d(y,z) \big)\\
    & \leq K_X \big( C_2\, q^{\lev A} + d(A,z) \big)\\
    & \leq K_X \big( C_2\, q^{\lev A} + rq^{\lev A} \big).
    \qedhere
  \end{align*}
\end{proof}

\begin{lem}\label{lem:diamond-intersection}
  Let $A_1, A_2 \in \dcubes$
  and assume that
  \begin{equation*}
    (r_1 \diamond A_1) \isect (r_2 \diamond A_2) \neq \emptyset,
  \end{equation*}
  for some $r_1, r_2 > 0$.
  Then
  \begin{equation*}
    r_2 \diamond A_2 \subset r \diamond A_1,
  \end{equation*}
  where
  $r = 2\, K_X^3 (C_2 + r_2 )\, q^{\lev A_2 - \lev A_1} + K_X\, r_1$.
\end{lem}

\begin{proof}
  Let $y \in (r_1 \diamond A_1) \isect (r_2 \diamond A_2)$ and
  $z \in (r_2 \diamond A_2)$, then
  \begin{equation*}
    d(z,A_1) \leq K_X\, \big( d(z,y) + d(y,A_1) \big).
  \end{equation*}
  Note that $d(y,A_1) \leq r_1\, q^{\lev A_1}$ and observe
  \begin{align*}
    d(z,y)
    &\leq K_X \big( d(z,m_{A_2}) + d(m_{A_2},y) \big)\\
    &\leq 2\, K_X^2( C_2 + r_2 )\, q^{\lev A_2},
  \end{align*}
  where we used Lemma~\ref{lem:diamond} for the latter estimate.
  Combining our estimates yields
  \begin{equation*}
    d(z,A_1) \leq K_X\, \big(
      2\, K_X^2 (C_2 + r_2 )\, q^{\lev A_2 - \lev A_1}  + r_1
    \big)\, q^{\lev A_1},
  \end{equation*}
thus the assertion of the lemma follows.
\end{proof}

\section{Adaptive dyadic grids}\label{s:one-third-shift}
\noindent
In this section we provide a customizable way to adapt dyadic grids, which is then applied in
Section~\ref{s:applications} to estimate the rearrangement operators $T_m$.

We recall that $K_X$, $C_d$ are constants defined by the quasimetric and the space of homogeneous
type $X$ and $C_1$, $C_2$ are determined by the collection of dyadic cubes,
cf. Section~\ref{s:preliminaries}.
For a given collection $\scr A$ of dyadic cubes in $X$ and $\alpha > 0$ we define
\begin{equation}\label{eq:def-scrA}
  \scr A^{(\alpha)}:=\Union_{A \in \scr A} \big\{
    Q\in\scr Q_{\lev A}\, :\,
    Q\isect \alpha \diamond A \neq \emptyset
  \big\}.
\end{equation}

\begin{figure}[bt]
  \begin{center}
    \begin{pspicture}(0,0)(10,-7)

      \psframe[linestyle=solid](0,0)(2,-2)
      \uput{0}[90]{0}(1,-1){$A_1$}
      \psarc[linecolor=gray](0,0){3.5}{270}{0}
      \psline[linecolor=gray,arrowsize=2pt 2]{->}(.5,-2)(.5,-3.464101)
      \uput[180]{0}(.5,-3){$\frac{C_R}{2 K_X}$}
      \psarc(0,0){4.5}{270}{0}
      \psline[arrowsize=2pt 2]{->}(2,-1)(4.387482,-1)
      \uput[90]{0}(4,-1){$C_R$}
      \psarc[linestyle=dashed](0,0){7}{270}{0}
      \psline[linestyle=dashed,arrowsize=2pt 2]{->}(1,-2)(1,-6.928203)
      \uput[180]{0}(1,-6.6){$\alpha$}
      \rput(1.5,-1.5){%
        \psccurve[linecolor=gray,fillstyle=solid,fillcolor=lightgray]%
        (0,0)(1,-.2)(.8,-.6)(0,-.7)
        \uput[90]{0}(.45,-.7){$\varphi(A_1)^*$}
      }

      \psframe[linestyle=solid](8,-2)(10,0)
      \uput{0}[90]{0}(9,-1){$A_2$}
      \psarc[linecolor=gray](10,0){3.5}{180}{270}
      \psline[linecolor=gray,arrowsize=2pt 2]{->}(9.5,-2)(9.5,-3.464101)
      \uput[0]{0}(9.5,-3){$\frac{C_R}{2 K_X}$}
      \psarc(10,0){4.5}{180}{270}
      \psline[arrowsize=2pt 2]{->}(8,-1)(5.612518,-1)
      \uput[90]{0}(6,-1){$C_R$}
      \psarc[linestyle=dashed](10,0){7}{180}{270}
      \psline[linestyle=dashed,arrowsize=2pt 2]{->}(9,-2)(9,-6.928203)
      \uput[0]{0}(9,-6.6){$\alpha$}
      \rput(7.5,-1.5){%
        \psccurve[linecolor=gray,fillstyle=solid,fillcolor=lightgray]%
        (0,0)(1,-.2)(.8,-.6)(0,-.7)
        \uput[90]{0}(.45,-.7){$\varphi(A_2)^*$}
      }

      \rput(1.4,-4.5){%
        \psscalebox{.25}{%
          \psframe[linestyle=solid](0,0)(2,-2)
          \psarc[linecolor=gray](0,0){3.5}{270}{0}
          \psarc(0,0){4.5}{270}{0}
          \psarc[linestyle=dashed](0,0){7}{270}{0}
          \rput(1.5,-1.5){%
            \psccurve[linecolor=gray,fillstyle=solid,fillcolor=lightgray]%
            (0,0)(1,-.2)(.8,-.6)(0,-.7)
          }

          \psframe[linestyle=solid](8,-2)(10,0)
          \psarc[linecolor=gray](10,0){3.5}{180}{270}
          \psarc(10,0){4.5}{180}{270}
          \psarc[linestyle=dashed](10,0){7}{180}{270}
          \rput(7.5,-1.5){%
            \psccurve[linecolor=gray,fillstyle=solid,fillcolor=lightgray]%
            (0,0)(1,-.2)(.8,-.6)(0,-.7)
          }
        }
        \uput[90]{0}(.25,-.55){$A_3$}
        \uput[90]{0}(2.25,-.55){$A_4$}
      }

    \end{pspicture}
  \end{center}
  \caption{\,}
  \label{fig:one-third-trick-1}
\end{figure}
The following result is a version of the well known one--third--trick in spaces of homogeneous
type.
\begin{thm}\label{thm:one-third-trick}
  Let $C_R > 0$ and $\mu \in \bb N$ be such that
  \begin{equation}\label{eq:constant-constraints}
    4\, K_X^3 (1 + C_2/C_R)\cdot q^\mu \leq 1.
  \end{equation}
  Let $\scr A \subset \scr Q$ be a finite collection of cubes satisfying
  \begin{enumerate}
  \item the separation condition
    \begin{equation}\label{eq:scrA-sparsity}
      \big( C_R \diamond A_1 \big) \isect
      \big( C_R \diamond A_2 \big)
      = \emptyset
    \end{equation}
    for all $A_1\neq A_2$ in $\scr A$ with $\lev A_1=\lev A_2$,
  \item the small successor condition
    \begin{equation}\label{eq:successor-property}
      \lev A \geq \mu + \lev \pred(A),
      \qquad \text{$A \in \scr A^{(\alpha)}$},
    \end{equation}
    where $\alpha = 2\, K_X^3 (C_2 + C_R) + C_R/2$ and $\pred \equiv \pred_{\scr A^{(\alpha)}}$.
  \end{enumerate}
  Let $\varphi\, :\, \scr A \to \pow(\scr A)$ be a map such that
  \begin{align}
    \lev Q & > \lev A,
    &&\text{$A \in \scr A$ and $Q \in \varphi(A)$},
    \label{eq:phi-condition-1}\\
    \varphi(A)^* & \subset \frac{C_R}{2\,  K_X}\diamond A,
    &&\text{$A \in \scr A$}.
    \label{eq:phi-condition-2}
  \end{align}

  Then there exist a collection $\scr B$ of adapted cubes in $X$ and a bijective map
  $\sigma:\scr A \rightarrow \scr B$ satisfying
  \begin{align}
    A\union \sigma(\varphi(A))^* \subset \sigma(A)\subset C_R\diamond A,
    && A &\in \scr A.
    \label{eq:sigma-property-1}
    \intertext{and the measure estimate}
    |\sigma(A)| \leq C_d \Big(\frac{K_X(C_2+C_R)}{C_1}\Big)^{\log_2(C_d)}\cdot |A|,
    && A &\in \scr A.
    \label{eq:sigma-property-2}
  \end{align}
  Moreover, the collection
  \begin{equation}\label{eq:ots-nested}
    \scr B = \big\{ \sigma(A)\, :\, A \in \scr A\big\}
  \end{equation}
  is nested.
\end{thm}
\noindent
The hypotheses of Theorem~\ref{thm:one-third-trick} are visualized in
Figure~\ref{fig:one-third-trick-1}.

\begin{myproof}
  We set $\widetilde{\scr A_j} := \scr A \isect \dcubes_j$, $j \in \bb Z$.
  Let the sequence $j_\ell$ be such that $\widetilde{\scr A_{j_\ell}} \neq \emptyset$ and
  $\widetilde{\scr A_k} = \emptyset$ for all $j_{\ell-1} < k < j_{\ell}$, $\ell \leq 0$.
  Then define $\scr A_\ell := \widetilde{\scr A_{j_\ell}}$, $\ell \leq 0$ and assume without
  restriction that $\scr A_0$ consists of the cubes in $\scr A$ with maximal level.
  The proof proceeds by induction on $\lev A$ for cubes $A \in \scr A$, starting with cubes in
  $\scr A_0$.
  \begin{proofstep}\label{proof:thm:one-third-trick:step-1}
    We begin the induction by defining
    \begin{equation*}
      \sigma(A) := A
      \text{ for $A \in \scr A_{0}$}
      \qquad\text{and}\qquad
      \scr B_{0} := \big\{ \sigma(A)\, :\, A \in \scr A_{0} \big\}.
    \end{equation*}
    Observe that~\eqref{eq:sigma-property-1} holds for all $A\in\scr A_0$.
    Now, let $k < 0$ and assume that all the cubes $\sigma(A)$, $A \in \scr A_{j}$, and the
    collections $\scr B_j := \big\{ \sigma(A)\, :\, A \in \scr A_j \big\}$ are already defined for
    all $j > k$.
    In order to construct $\sigma(A)$, let $A \in \scr A_k$ and define
    \begin{equation}\label{eq:sigma-induction-2}
      \sigma(A) := \big( A \union \sigma(\varphi(A))^* \big) \union
      \Big( \Union \big\{ B \in \scr B_j\, :\,
        j > k,B \isect (A \union \sigma(\varphi(A))^*) \neq \emptyset
      \big\}\Big).
    \end{equation}
    We collect all those cubes in
    \begin{equation*}
      \scr B_k := \big\{ \sigma(A)\, :\, A \in \scr A_k \big\}.
    \end{equation*}
    Finally, the set $\scr B$ of all adapted cubes is defined as
    \begin{equation*}
      \scr B := \Union_j \scr B_j.
    \end{equation*}
    In the next two steps we will inductively verify the nestedness of $\scr B$ and the localization
    property~\eqref{eq:sigma-property-1}.
  \end{proofstep}

  \begin{proofstep}\label{proof:thm:one-third-trick:step-2}
    Here we prove the nestedness of $\scr B$.
    To this end, define the level of an adapted cube $B = \sigma(A)$ by $\lev B = \lev A$.
    Let $B_1,\, B_2 \in \scr B$ be such that $B_1 \isect B_2 \neq \emptyset$ and assume
    $\lev B_1 \leq \lev B_2$.
    If $\lev B_1 = \lev B_2$, then properties~\eqref{eq:scrA-sparsity}
    and~\eqref{eq:sigma-property-1} yield $B_1 = B_2$.
    So we may now assume that $\lev B_1 < \lev B_2$.
    Choose $A_1 \in \scr A$ such that $\sigma(A_1) = B_1$.
    If $B_2 \isect \big( A_1 \union \sigma(\varphi(A_1))^* \big) = \emptyset$ we get
    $B_1 \isect B_2 = \emptyset$ by definition of $B_1$, cf.~\eqref{eq:sigma-induction-2}.
    This contradicts the assumption $B_1 \isect B_2 \neq \emptyset$.
    Thus,
    $B_2 \isect \big( A_1 \union \sigma(\varphi(A))^* \big) \neq \emptyset$
    and, by~\eqref{eq:sigma-induction-2} again, we infer $B_2 \subset B_1$,
    proving the nestedness of $\scr B$.
  \end{proofstep}

  \begin{proofstep}\label{proof:thm:one-third-trick:step-3}
    In this step we will verify~\eqref{eq:sigma-property-1}.
    Assume that~\eqref{eq:sigma-property-1} is true for all $A \in \scr A_j$, $j > k$.
    Recall that $\scr B$ is nested by Step~\ref{proof:thm:one-third-trick:step-2} of this proof.
    Now, let $A \in \scr A_k$ be fixed.
    First, note that $A \union \sigma(\varphi(A))^* \subset \sigma(A)$ by
    the definition of $\sigma(A)$, cf.~\eqref{eq:sigma-induction-2}.
    Secondly, we show that $\sigma(A) \subset C_R\diamond A$.
    Let $B \in \scr B_j$, $j > k$ be such that
    $B \isect (A \union \sigma(\varphi(A))^*) \neq \emptyset$.
    The condition $B \isect (A \union \sigma(\varphi(A))^*) \neq \emptyset$ is covered by the cases
    \begin{enumerate}
    \item $B \isect \frac{C_R}{2\, K_X}\diamond A \neq \emptyset$,
      \label{enu:diamond}
    \item there exists a $Q \in \varphi(A)$ such that
      $B \isect \sigma(Q) \neq \emptyset$, and so by~\eqref{eq:ots-nested} either
      \label{enu:nested}
      \begin{enumerate}
      \item $\sigma(Q) \subset B$ or
        \label{enu:nested-1}
      \item $B \subset \sigma(Q)$.
        \label{enu:nested-2}
      \end{enumerate}
    \end{enumerate}
    First, let us consider case~\eqref{enu:diamond}.
    Due to the induction hypothesis, \eqref{eq:sigma-property-1} is true for $\sigma^{-1}(B)$, that
    is $B \subset C_R\diamond \sigma^{-1}(B)$.
    Thus, Lemma~\ref{lem:diamond-intersection} implies
    \begin{equation}\label{eq:ots-localization-1}
      B \subset C_R\diamond \sigma^{-1}(B) \subset r \diamond A,
    \end{equation}
    where
    $r = 2\, K_X^3 (C_2 + C_R)\cdot q^{\lev \sigma^{-1}(B)- \lev A} + C_R/2$.
    Observe that since $r \leq \alpha$, we can find a cube $\widetilde A \in \scr A_k^{(\alpha)}$
    such that $\sigma^{-1}(B) \subsetneq \widetilde A$.
    Hence $\lev \sigma^{-1}(B) \geq \mu + \lev \widetilde A = \mu + \lev A$
    and so
    $r \leq 2\, K_X^3 (C_2 + C_R)\cdot q^\mu + \frac{C_R}{2}$.
    Since $r \leq C_R$ by~\eqref{eq:constant-constraints}, the inclusion $B \subset C_R\diamond A$
    follows.

    In case~\eqref{enu:nested-1}, the first inclusion in~\eqref{eq:sigma-property-1} yields
    $Q \subset \sigma(Q) \subset B$.
    Since $Q \subset \varphi(A)^* \subset \frac{C_R}{2\, K_X}\diamond A$
    by~\eqref{eq:phi-condition-2}, in particular
    $B \isect \frac{C_R}{2\, K_X}\diamond A \neq \emptyset$.
    Hence, case~\eqref{enu:nested-1} is covered by case~\eqref{enu:diamond}.
    In case~\eqref{enu:nested-2}, condition~\eqref{eq:phi-condition-2} implies
    $\sigma(Q) \isect \frac{C_R}{2\, K_X}\diamond A \neq \emptyset$.
    Applying the proof of case~\eqref{enu:diamond} to $\sigma(Q)$ instead of $B$, we obtain
    $\sigma(Q) \subset C_R \diamond A$, and thus $B \subset C_R \diamond A$.
 
    To summarize, in any of the cases~\eqref{enu:diamond}, \eqref{enu:nested-1}
    and~\eqref{enu:nested-2}, the condition
    $B \isect (A \union \sigma(\varphi(A))^*) \neq \emptyset$
    yields $B \subset C_R\diamond A$, which proves~\eqref{eq:sigma-property-1}, i.e.,
    $\sigma(A) \subset C_R\diamond A$.

    Finally, the measure estimate~\eqref{eq:sigma-property-2} is an immediate consequence of the
    doubling condition~\eqref{eq:doubling} and
    \begin{equation*}
      B(m_A, C_1q^{\lev A})
      \subset A \subset \sigma(A) \subset C_R\diamond A
      \subset B(m_A, K_X(C_2+C_R)q^{\lev A}),
    \end{equation*}
    where the latter inclusion follows from Lemma~\ref{lem:diamond}.
    \qedhere
  \end{proofstep}
\end{myproof}

\section{Rearrangement operators}\label{s:applications}

Following~\cite{MP2012}, we define and analyze rearrangement operators on spaces of homogeneous
type, thereby extending the rearrangement operators $T_m$ introduced in~\cite{Figiel1988}, that act
on the standard Haar system.

\bfpar{The shift relation $\tau$}
Let $m \in \bb R$, $m > 0$ and
$\tau \subset \Union\limits_{j \in \bb Z} \dcubes_j \times \dcubes_j$
have the properties
\begin{enumerate}[\indent(P1)]
\item $Q \subset m \diamond P$ for all $(P,Q) \in \tau$, cf. Figure~\ref{fig:shift-setup-1},
  \label{enu:P1}
\item there exists a finite partition $\tau_1,\dots,\tau_M$ of $\tau$ such that
  $\tau_k$ is a bijective function for all $1 \leq k \leq M$.
  \label{enu:P2}
\end{enumerate}
The relation $\tau$ generalizes the classical shift $I\mapsto I+m|I|$ on $\mathbb R$,
cf.~\cite{Figiel1988}.

In order to apply Theorem \ref{thm:one-third-trick} to our shift $\tau$, we decompose $\tau_k$ into
suitable subcollections in the following way.
\begin{enumerate}[\indent(C1)]
\item First, let us choose a constant $C_R > 0$ and split $\tau_k$ into the collections
  $\scr G_{k,1},\dots,\scr G_{k,M_k}$, for all $1 \leq k \leq M$ such that
  \begin{equation}\label{eq:c1}
    \big( C_R\diamond \tau_k^i(A_1) \big)
    \isect
    \big( C_R\diamond \tau_k^n(A_2) \big)
    = \emptyset
  \end{equation}
  for all $A_1, A_2 \in \proj_1(\scr G_{k,j})$, $1 \leq j \leq M_k$, $i,n \in \{0,1\}$, where
  $\tau_k^i(A)$ is defined to be $A$ for $i=0$ and $\tau_k(A)$ for $i=1$.
  The projection onto the first and second coordinates of a relation are denoted by $\proj_1$ and
  $\proj_2$, respectively.
  Observe that the constants $M_k, 1\leq k\leq M$ depend only on $X$, cf.~\cite{MP2012}.
  We refer to Figure~\ref{fig:shift-setup-1} for a picture of the separation
  condition~\eqref{eq:c1}.
  \label{enu:C1}
\item Secondly, let $\ell$ be a positive integer and define
  \begin{equation*}
    \scr H_{k,j,i}^{(\ell)}
    = \scr G_{k,j} \isect \Big(
      \Union_{r \in \bb Z} \dcubes_{i + r\, \ell} \times \dcubes_{i + r\, \ell} 
    \Big),
  \end{equation*}
  for all $1 \leq k \leq M$, $1 \leq j \leq M_k$, $0 \leq i \leq \ell -1$.
  The parameter $\ell$ will later be chosen to be approximately $\log (2+m)$ with $m$ being the
  parameter from (P\ref{enu:P1}).
  \label{enu:C2}
\item Finally, define
  \begin{equation*}
    \psi_{k,j,i}^{(\ell)}(A) = \big\{ (P,Q) \in \scr H_{k,j,i}^{(\ell)}\, :\,
      \text{$P \subsetneq A$ or $Q \subsetneq A$}
    \big\},
  \end{equation*}
  for all $A \in \big( \proj_1(\scr H_{k,j,i}^{(\ell)}) \union \proj_2(\scr H_{k,j,i}^{(\ell)}) \big)$.
  \label{enu:C3}
\end{enumerate}

\begin{figure}[bt]
  \begin{center}
    \begin{pspicture}(0,0)(10,-7)

      \psframe[linestyle=solid](0,0)(.5,-.5)
      \uput{0}[0]{0}(.5,-.25){$P$}
      \psarc[linecolor=gray](0,0){5.5}{270}{0}
      \psline[linecolor=gray,arrowsize=2pt 2]{->}(.5,-.5)(3.889088,-3.889088)
      \uput[225]{0}(3.25,-3.25){{\gray $m$}}
      \psarc(.25,-.25){3.5}{265}{5}
      \psline{->}(.25,-.5)(.25,-3.75)
      \uput[0]{0}(.25,-3.25){$C_R$}
      \rput(2,-3.5){
        \psframe[linestyle=solid](0,0)(.5,-.5)
        \uput{0}[0]{0}(.5,-.25){$Q$}
        \psarc[linestyle=dashed](.25,-.25){3.5}{250}{80}
        \psline[linestyle=dashed]{->}(.5,-.5)(2.724874,-2.724874)
        \uput[225]{0}(2,-2){$C_R$}
      }

      \psframe[linestyle=solid](9.5,0)(10,-.5)
      \uput{0}[225]{0}(9.5,-.5){$P'$}
      \psarc[linecolor=gray](10,0){5.5}{180}{270}
      \psline[linecolor=gray,arrowsize=2pt 2]{->}(9.75,-.5)(9.75,-5.5)
      \uput[180]{0}(9.75,-4.5){{\gray $m$}}
      \psarc(10,0){3.5}{180}{270}
      \psline{->}(9.5,-.25)(6.5,-.25)
      \uput[270]{0}(7,-.25){$C_R$}
      \rput(-.5,-1.5){
        \psframe[linestyle=solid](9.5,0)(10,-.5)
        \uput{0}[180]{0}(9.375,-.25){$Q'$}
        \psarc[linestyle=dashed](10,0){3.5}{155}{278}
        \psline[linestyle=dashed]{->}(9.5,-.5)(7.5,-2.5)
        \uput[0]{0}(8,-2.25){$C_R$}
      }%
    \end{pspicture}
  \end{center}
  \caption{\,}
  \label{fig:shift-setup-1}
\end{figure}

\noindent
The collections $\psi_{k,j,i}^{(\ell)}(A)$ are well localized around $A$, which is discussed in
\begin{lem}\label{lem:successor-localization}
  Let $m \in \bb R$, $m > 0$ and let $\ell$ be a positive integer.
  Then
  \begin{equation*}
    \big( \proj_1(\psi_{k,j,i}^{(\ell)}(A)) \union \proj_2(\psi_{k,j,i}^{(\ell)}(A)) \big)^*
    \subset \big( c_1\, (1 + m)\, q^\ell \big) \diamond A ,
  \end{equation*}
  for all $A \in \scr H_{k,j,i}^{(\ell)}$,
  $1 \leq k \leq M$, $1 \leq j \leq M_k$,
  $0 \leq i \leq \ell -1$.
  The constant $c_1$ depends only on the space of homogeneous type $X$.
\end{lem}
\begin{proof}
  Let $(P,Q)\in \psi_{k,j,i}^{(\ell)}(A)$.
  Then, $P \subsetneq A$ or $Q \subsetneq A$ by definition of $\psi_{k,j,i}^{(\ell)}$.
  We know from  (P\ref{enu:P1}) that
  $P \union Q \subset m\diamond P$, hence Lemma~\ref{lem:diamond-intersection} yields
  \begin{equation*}
    P \union Q \subset \big( 2\, K_X^3 ( C_2 + m )q^{\lev P - \lev A} \diamond A \big).
  \end{equation*}
  Noting that $\lev P \geq \ell+\lev A$ by (C\ref{enu:C2}) concludes the proof.
\end{proof}

\bfpar{The Shift operator $T$}
In order to define analogues of $T_m$ on spaces of homogeneous type, we need a substitute $\{h_Q\}$
for the standard Haar system.
We require the system of functions $\{h_Q\}_{Q\in\scr Q}$ to satisfy the conditions
\begin{enumerate}[\indent(H1)]
\item $\supp h_Q \subset Q$, for all $Q \in \dcubes$,
  \label{enu:H1}
\item $\big\|h_Q\big\|_\infty \leq C_h \frac{1}{|P| + |Q|}\int |h_P|$,
  for all $(P,Q) \in \tau$,
  \label{enu:H2}
\item for each $k$ the collections
  $\big\{ h_P\, :\, P \in \proj_1(\tau_k) \big\}$
  and
  $\big\{ h_Q\, :\, Q \in \proj_2(\tau_k) \big\}$
  constitute martingale difference sequences, separately.
  \label{enu:H3}
\end{enumerate}
The constant $C_h > 0$ is independent of $(P,Q)$.
The collections $\scr H_{k,j,i}^{(\ell)}$, defined in~(C\ref{enu:C2}), naturally induce the
subspaces $H_{k,j,i}^{(\ell)}$ of $L_E^p(X)$ given by
\begin{equation*}
  H_{k,j,i}^{(\ell)} = \Big\{
  f \in L_E^p(X)\, :\,
  f = \sum_{P \in \proj_1(\scr H_{k,j,i}^{(\ell)})} \langle f, h_P \rangle\, h_P
  \Big\}.
\end{equation*}
We now define the shift operators $T_k$ induced by $\tau_k$, $1 \leq k \leq M$, as the linear
extension of the map
\begin{equation}\label{eq:definition-T_k}
  h_P \mapsto
  \begin{cases}
    h_Q,& \text{if $(P,Q) \in \tau_k$},\\
    0,& \text{otherwise}.
  \end{cases}
\end{equation}
If the collections $\psi_{k,j,i}^{(\ell)}$ are sufficiently localized, then the operators $T_k$ are
bounded on the subspace $H_{k,j,i}^{(\ell)}$.
The details are given in the theorem below.
\begin{thm}\label{thm:shift-1}
  Let $X$ be a space of homogeneous type, $E$ a $\umd$--space and
  $1 < p < \infty$.
  Let $m \in \bb R$, $m > 0$, then there exists a constant $\beta > 0$ such that for all integers
  $\ell$ satisfying
  \begin{equation}\label{eq:ell-condition}
    (1 + m)\, q^\ell \leq \beta,
  \end{equation}
  we have
  \begin{equation}\label{eq:figiel-shift-estimate-1}
    \big\| T_k f \big\|_{L_E^p(X)} \leq C\, \|f\|_{L_E^p(X)},
    \qquad \text{$f \in H_{k,j,i}^{(\ell)}$},
  \end{equation}
  for all $1 \leq k \leq M$, $1 \leq j \leq M_k$, $0 \leq i \leq \ell -1$.
  The constant $C$ depends only on $p$, $X$ and $E$, and the constant $\beta$
  only on $X$.
\end{thm}

\begin{proof}
  Let $\ell$ be fixed throughout the proof satisfying~\eqref{eq:ell-condition}.
  Conditions on the constant $\beta$ will be imposed within the proof.

  Our goal is to apply Theorem~\ref{thm:one-third-trick} to each of the collections
  $\scr H_{k,j,i}^{(\ell)}$.
  With $k,j,i$ fixed, let us define the collections
  $\scr C = \scr C^{(1)} \union \scr C^{(2)} = \proj_1(\scr H_{k,j,i}^{(\ell)}) \union \proj_2(\scr H_{k,j,i}^{(\ell)})$
  and let
  \begin{equation*}
    \scr A \subset \scr C
    \quad \text{be a finite set}.
  \end{equation*}
  The function $\varphi\, :\, \scr A \to \pow(\scr A)$ is given by
  \begin{equation*}
    \varphi(A) := \proj_1(\psi_{k,j,i}^{(\ell)}(A)) \union \proj_2(\psi_{k,j,i}^{(\ell)}(A)),
    \qquad\text{$A \in \scr A$},
  \end{equation*}
  where $\psi_{k,j,i}^{(\ell)}$ is defined in~(C\ref{enu:C3}).
  We shall now verify that $\scr A$ and $\varphi$ satisfy the hypotheses of
  Theorem~\ref{thm:one-third-trick}.

  First, observe that the separation condition~\eqref{eq:scrA-sparsity} is satisfied due
  to~(C\ref{enu:C1}).
  Secondly, let $\mu = \ell$, then~\eqref{eq:constant-constraints} holds for sufficiently
  small $\beta$, where the constraint for $\beta$ depends only on $X$.
  Additionally, observe that~(C\ref{enu:C2}) implies~\eqref{eq:successor-property}.
  From Lemma~\ref{lem:successor-localization} and~\eqref{eq:ell-condition} it follows that
  $\varphi(A)^* \subset \frac{C_R}{2\, K_X} \diamond A$ if $\beta$ is sufficiently small.
  Having verified all the hypotheses of Theorem~\ref{thm:one-third-trick}, we
  obtain a nested collection of sets $\scr B$ and a bijective map
  $\sigma\, :\, \scr A \to \scr B$, such that
  \begin{equation*}
    A \union \sigma(\varphi(A))^* \subset \sigma(A)\subset C_R\diamond A
    \qquad \text{and}\qquad
    |\sigma(A)| \leq c_2 (1+C_R)^{\log_2(C_d)}\cdot |A|
  \end{equation*}
  for all $A\in \scr A$.
  The constant $c_2$ depends only on $X$.

  Let us now define a nested collection of sets supporting the shifts $\tau$ inductively,
  beginning with the smallest cubes.
  Set $n_{\max} = \max\{\lev(A)\, :\, A \in \scr A\}$ and define
  \begin{equation*}
    \theta(P) := \theta(Q) := \sigma(P) \union \sigma(Q)
  \end{equation*}
  for all $(P,Q) \in \scr H_{k,j,i}^{(\ell)}$ such that $\lev(P) = n_{\max}$.
  With $n < n_{\max}$ fixed, assume that $\theta(A)$ is already defined for all cubes
  $A$ satisfying $\lev(A) > n$.
  The function $\theta$ is specified by
  \begin{equation*}
    \theta(P) := \theta(Q) := \sigma(P) \union \sigma(Q) \union
    \big\{ \theta(R)\, :\,
      \text{$\lev R > \lev P$,
        $\theta(R) \isect (\sigma(P) \union \sigma(Q)) \neq \emptyset$}
    \big\}^*,
  \end{equation*}
  for all $(P,Q) \in \scr H_{k,j,i}^{(\ell)}$ with $\lev(P) = n$.
  As an immediate consequence of the principle of construction, the collection
  $\{ \theta(A)\, :\, A \in \scr A\}$
  is nested and 
  \begin{equation*}
    P \union Q \subset \theta(P) = \theta(Q),
    \qquad\text{$(P,Q) \in \scr H_{k,j,i}^{(\ell)}$, $P \in \scr A$}.
  \end{equation*}
  Furthermore, a straightforward calculation using Lemma~\ref{lem:diamond-intersection}
  and~\eqref{eq:ell-condition} shows that there exists a
  constant $c_3$ depending only on $X$ such that
  \begin{equation*}
    \theta(P) \subset (c_3\diamond P) \union (c_3 \diamond Q),
    \qquad \text{$(P,Q) \in \scr H_{k,j,i}^{(\ell)}$, $P \in \scr A$},
  \end{equation*}
  if $\beta$ is sufficiently small.
  From the latter inclusion we obtain
  \begin{equation}\label{eq:theta-estimate}
    \theta(P) \leq c_4\, (|P| + |Q|),
    \qquad \text{$(P,Q) \in \scr H_{k,j,i}^{(\ell)}$, $P \in \scr A$},
  \end{equation}
  where $c_4$ depends only on $X$.
  Let us define the filtration $\{ \cal F_{n} \}$ by
  \begin{equation*}
    \cal F_{n} =
    \salg\Big( \big\{ \theta(A)\, :\, A \in \scr A, \lev A \leq n \big\}\Big),
    \qquad\text{$n \in \bb Z$}.
  \end{equation*}
  Observe that $\theta(A)$ is an atom in $\cal F_{\lev A}$ for all $A \in \scr A$, since
  $\big\{ \theta(A)\, :\, A \in \scr A\big\}$ is a nested collection.
  Thus, (H\ref{enu:H2}) and~\eqref{eq:theta-estimate} imply
  \begin{equation}\label{eq:mds-estimate}
    \big| h_{\tau_k(A)} \big| \leq c_5 \cond\big(|h_A|\, \big|\, \cal F_n\big),
    \qquad\text{$A \in \scr E_n$},
  \end{equation}
  where $\scr E_n = \dcubes_n \isect \scr A \isect \scr C^{(1)}$ and $c_5$ depends only on $X$ and
  $C_h$.

  We will now estimate $Tf$ for all $f \in H_{k,j,i}^{(\ell)}$.
  Note that (H\ref{enu:H3}) and the $\umd$--property of $E$ allow us to assume that $f$ is of the
  form
  \begin{equation*}
    f = \sum_{n} \sum_{A \in \scr E_n} \langle f, h_A \rangle\, h_A.
  \end{equation*}
  Moreover, $T|_{H_{k,j,i}^{(\ell)}} = T_k|_{H_{k,j,i}^{(\ell)}}$ is a function due to~(P\ref{enu:P2}).
  By employing (H\ref{enu:H3}) again, we introduce Rademacher means in
  $\|T_k f\|$ and obtain
  \begin{align*}
    \big\| T_k f \big\|
    & = \Big\| \sum_n \sum_{A \in \scr E_n}
      \langle f, h_A \rangle\, h_{\tau_k(A)} \Big\|\\
    & \approx \int\limits _0^1 \Big\| \sum_n r_n(t)
      \sum_{A \in \scr E_n}\, \langle f, h_A \rangle\, h_{\tau_k(A)} \Big\|\dif t.
  \end{align*}
  Furthermore, estimate~\eqref{eq:mds-estimate} yields
  \begin{align*}
    \big\| T_k f \big\|
    &\approx \int\limits _0^1 \Big\| \sum_n r_n(t) \sum_{A \in \scr E_n}
      \langle f, h_A \rangle\, |h_{\tau_k(A)}| \Big\|\dif t\\
    &\lesssim \int\limits _0^1 \Big\| \sum_n r_n(t) \cond\Big(
      \sum_{A \in \scr E_n}
        \langle f, h_A \rangle\, |h_A|\ \Big|\, \cal F_n \Big)
      \Big\|\dif t,
  \end{align*}
  by means of Kahane's contraction principle.
  Applying Bourgain's version of Stein's martingale inequality gives us
  \begin{equation*}
    \big\| T_k f \big\|
    \lesssim \int\limits _0^1 \Big\| \sum_n r_n(t)
      \sum_{A \in \scr E_n} \langle f, h_A \rangle\, |h_A|
    \Big\|\dif t.
  \end{equation*}
  Using Kahane's contraction principle and the $\umd$--property, cf. (H\ref{enu:H3}), concludes the
  proof.
\end{proof}

Combining the estimates of Theorem~\ref{thm:shift-1} on the subspaces $H_{k,j,i}^{(\ell)}$, we
obtain estimates for $T_k$ on $\overline{\lin}\{h_P:P\in \scr Q\}$ in the subsequent theorem,
cf.~\cite{Figiel1988}.
\begin{thm}\label{thm:shift-2}
  Let $X$ be a space of homogeneous type, $E$ a $\umd$--space, $1 < p < \infty$ and
  $m \in \bb R$, $m > 0$.
  Then for all $1 \leq k \leq M$ the linear operator $T_k$ satisfies
  \begin{equation}\label{eq:figiel-shift-estimate-2}
    \big\| T_k f \big\|_{L_E^p(X)} \leq C \log(2 + m)^\alpha\, \|f\|_{L_E^p(X)},
    \qquad \text{$f \in \overline{\lin}\{h_P:P\in \scr Q\}$}.
  \end{equation}
  If $L_E^p(X)$ has type $\cal T$ and cotype $\cal C$, then $\alpha < 1$ is given by
  $1/\cal T - 1/\cal C$.
  The constant $C$ depends only on $p$, $X$, $E$ and $\alpha$.
\end{thm}

\begin{proof}
  Within this proof we shall abbreviate $\|\cdot \|_{L_E^p(X)}$ by $\|\cdot\|$.
  Let $m>0$ and choose $\ell$ as the minimal integer satisfying~\eqref{eq:ell-condition}, i.e.,
  there exists a constant $c_1$ only depending on $X$ with
  \begin{equation}\label{eq:defl}
    \ell \geq c_1\log(2+m).
  \end{equation}
  Assume that $f$ is a finite sum of the form
  $\displaystyle f
  = \sum_{j=1}^{M_k}\sum_{i=0}^{\ell-1} \sum_{P \in \scr H_{k,j,i}^{(\ell)}} f_P\, h_P$.
  Then, by definition of $T_k$ and the $\umd$--property of $L_E^p(X)$ applied to~(H\ref{enu:H3}), we
  obtain
  \begin{equation*}
    \| T_k f \|
    \lesssim \int_0^1 \Big\| \sum_{j,i} r_{j,i}(t)\, T_k\, d_{j,i}\Big\| \dif t,
  \end{equation*}
  where
  $d_{j,i} = \sum\limits_{P \in \scr H_{k,j,i}^{(\ell)}} f_P\, h_P$.
  The type inequality yields
  \begin{equation*}
    \| T_k f \| \lesssim
    \Big( \sum_{j,i} \| T_k\, d_{j,i} \|^{\cal T} \Big)^{1/\cal T},
  \end{equation*}
  where $L_E^p(X)$ is of type $\cal T$.
  Theorem~\ref{thm:shift-1} implies
  $\| T_k\, d_{j,i} \| \lesssim \| d_{j,i} \|$, hence
  \begin{equation*}
    \| T_k f \|
    \lesssim \big( M_k\cdot \ell \big)^{1/\cal T - 1/\cal C}
    \, \Big( \sum_{j,i} \| d_{j,i} \|^{\cal C} \Big)^{1/\cal C},
  \end{equation*}
  where $L_E^p(X)$ is of cotype $\cal C$.
  The cotype inequality and the $\umd$--property show
  \begin{equation*}
    \| T_k f \|
    \lesssim \big( M_k\cdot \ell \big)^{1/\cal T - 1/\cal C}
    \, \|f\|.
  \end{equation*}
  Since $M_k$ depends only on $X$, using \eqref{eq:defl} gives \eqref{eq:figiel-shift-estimate-2}
  for finite sums $f$ in $\lin\{h_P:P\in\scr Q\}$, thus concluding the proof by unique extension.
\end{proof}

\section{Stripe operator}\label{s:stripe-operator}

\noindent
In this section we define stripe operators on spaces of homogeneous type and provide vector--valued
$L^p$ estimates.
Our notion of stripe operators generalizes those on $\mathbb R^k$ analyzed
in~\cite{lechner:int:2011}, which will now be briefly reviewed.

For a positive integer $\lambda$, the stripes $\scr S_\lambda^{(m)}$ of the dyadic cube $[0,1]^n$
are given by
\begin{equation}\label{eq:stripe-original}
  \scr S_\lambda^{(m)}([0,1]^k)=
  \Bigg\{\, Q\, :\,
    \begin{array}{l}
      \text{$Q$ is a dyadic cube with $|Q|=2^{-\lambda k}$},\\
      Q \subset \Big[\frac{m-1}{2^\lambda},\frac{m}{2^\lambda}\Big]\times [0,1]^{k-1}
    \end{array}
  \Bigg\},
\end{equation}
where $1 \leq m \leq 2^\lambda$.
For an arbitrary dyadic cube $A$, the stripes $\scr S_\lambda^{(m)}(A)$ are obtained by scaling and
translating $\scr S_\lambda^{(m)}([0,1]^k)$ to the position of $A$ in the dyadic grid.
The stripe operators $S_\lambda^{(m)}$ are defined by
\begin{equation}\label{eq:stripe_operator-original}
  S_\lambda^{(m)} h_A := g_{A,\lambda}^{(m)} := \sum_{R \in \scr S_\lambda^{(m)}(A)} h_R,
\end{equation}
where $h_A$ and $h_R$ denote canonical Haar functions supported on the dyadic cubes $A$ and $R$.
Estimates for $S_\lambda^{(m)}$ on $L^p$ were used in~\cite{LeeMuellerMueller2011} as well as
in~\cite{lechner:int:2011} to show weak lower semi-continuity for functionals with separately
convex integrands on scalar and vector--valued $L^p$, respectively.

We will now extend the operators $S_\lambda^{(m)}$ and their vector--valued estimates to spaces of
homogeneous type.
\bfpar{The stripes $\scr S_\lambda^{(m)}$}
Let $\lambda$ and $M$ be positive integers and define the stripes
$\stripe_\lambda^{(m)}(A)$, $A \in \dcubes$, $1\leq m \leq M$
as arbitrary subsets of
$\big\{ B \subset A\, :\, \lev B = \lev A + \lambda\big\}$
satisfying the conditions
\begin{enumerate}[\indent(S1)]
\item $A = \Union_{m=1}^M \stripe_\lambda^{(m)}(A)^*$ is a disjoint union,
  \label{enu:S1}
\item \label{enu:S2}
  there exists an absolute constant $K_1$ such that
  \[
  |\stripe_\lambda^{(m)}(A)^*| \leq K_1\, |\stripe_\lambda^{(n)}(A)^*|,\qquad 1\leq m,n\leq M,
  \]
\item $\big\{ \stripe_\lambda^{(m)}(A)^* \, :\, A \in \dcubes \big\}$ is nested,
  with $1 \leq m \leq M$ being fixed,
  \label{enu:S3}
\item there exist constants $\varepsilon>0$ and $K_2$ depending only on $X$ such that
  for all $1 \leq m \leq M$ we have
  \begin{equation*}
    |\scr E_j^{(m)}(A)^*| \leq K_2 q^{j\varepsilon} |A|,
    \qquad 0\leq j\leq \lambda-1,
  \end{equation*}
  where
  \[
  \scr E_j^{(m)}(A) := \{B \in \dcubes_{\lev A+j} :  B\cap \stripe_\lambda^{(m)}(A)^*\neq \emptyset\}.
  \]
  \label{enu:S4}
\end{enumerate}
The classical stripe~\eqref{eq:stripe-original} defined in $\bb R^k$ equipped with the Euclidean
metric satisfies the conditions (S\ref{enu:S1}) to (S\ref{enu:S4}) with parameters
$M = 2^\lambda$, $K_1 = 1$, $K_2 = 1$, $q = 1/2$ and $\varepsilon = 1$.

\bfpar{The stripe operators $S_\lambda^{(m)}$}
Let the collection of functions $\big\{ h_A\, :\, A \in \dcubes \big\}$ suffice
\begin{enumerate}[\indent(M1)]
\item $\supp h_A \subset A$, for all $A \in \dcubes$,
  \label{enu:M1}
\item $\big\{ h_A\, :\, A \in \dcubes \big\}$ constitutes a martingale
  difference sequence.
  \label{enu:M2}
\end{enumerate}
Moreover, let $\big\{ g_{A,\lambda}^{(m)}\, :\, A \in \dcubes \big\}$, $1 \leq m \leq M$ be
collections of functions that satisfy
\begin{enumerate}[\indent(G1)]
\item $\supp g_{A,\lambda}^{(m)} \subset \stripe_\lambda^{(m)}(A)$,
  for all $A \in \dcubes$ and $1 \leq m \leq M$,
  \label{enu:G1}
\item $\big\{ g_{A,\lambda}^{(m)}\, :\, 1 \leq m \leq M, A \in \dcubes \big\}$
  constitutes a martingale difference sequence,
  \label{enu:G2}
\item $\big\| g_{A,\lambda}^{(m)} \big\|_\infty
  \leq C_g\, \frac{1}{|\stripe_\lambda^{(n)}(A)^*|} \int |g_{A,\lambda}^{(n)}|$, 
  for all $A \in \dcubes$, $1 \leq m, n \leq M$, $m \neq n$ and some constant $C_g\geq 1$.
  \label{enu:G3}
\end{enumerate}
We define the stripe operator $S_\lambda^{(m)}$, $1 \leq m \leq M$, as the linear extension of
\begin{equation}\label{eq:stripe-operator-dfn}
  S_\lambda^{(m)} h_A := g_{A,\lambda}^{(m)},\qquad A\in\scr Q.
\end{equation}
Note that the classical stripe operator~\eqref{eq:stripe_operator-original} satisfies all of the
above conditions.

\begin{lem}\label{lem:stripe-functions}
  Let $g_{A,\lambda}^{(m)}$ and $g_{A,\lambda}^{(n)}$ be stripe functions
  satisfying (G\ref{enu:G1}) and~(G\ref{enu:G3}), then
  \begin{equation*}
    \Big| \Big\{
      |g_{A,\lambda}^{(n)}| \geq \frac{\| g_{A,\lambda}^{(m)} \|_\infty}{2C_g}
    \Big\} \Big|
    \geq \frac{1}{2C_g^2}\, \big| \stripe_\lambda^{(n)}(A)^* \big|.
  \end{equation*}
\end{lem}

\begin{proof}
  We shall abbreviate $g^{(m)} = g_{A,\lambda}^{(m)}$,
  $g^{(n)} = g_{A,\lambda}^{(m)}$ and $\stripe=\stripe_\lambda^{(n)}(A)$.
  Assume the contrary, that is
  \begin{equation}\label{eq:stripe-functions-1}
    \Big| \Big\{
      |g^{(n)}| \geq \frac{\| g^{(m)} \|_\infty}{2C_g}
    \Big\} \Big|
    < \frac{1}{2C_g^2}\, \big| \stripe^* \big|.
  \end{equation}
  Then, (G\ref{enu:G3}) implies
  \begin{align*}
    \frac{\big| \stripe^* \big|}{C_g}\cdot \| g^{(m)} \|_\infty
    \leq \int\limits_{\stripe^*} |g^{(n)}|
    &\leq \Big|\Big\{
        |g^{(n)}| < \frac{\|g^{(m)}\|_\infty}{2C_g}
      \Big\}\Big|\cdot \frac{\|g^{(m)}\|_\infty}{2C_g}\\
    & \qquad+ \Big|\Big\{
        |g^{(n)}| \geq \frac{\|g^{(m)}\|_\infty}{2C_g}
      \Big\}\Big|\cdot \|g^{(n)}\|_\infty.
  \end{align*}
  Observe that~(G\ref{enu:G3}) and~(G\ref{enu:G1}) give us
  $\|g^{(n)}\|_\infty \leq C_g\, \|g^{(m)}\|$, thus
  inserting (G\ref{enu:G1}) and~\eqref{eq:stripe-functions-1} in the latter
  display yields a contradiction, proving the lemma.
\end{proof}

The subsequent results, i.e., the combinatorial Lemma~\ref{lem:stripe-combinatorial} and the
estimates on stripe operators Theorems~\ref{thm:stripe} and~\ref{thm:stripe-main} are proved by
similar methods as their Euclidean counterparts in~\cite{lechner:int:2011}.
\begin{lem}\label{lem:stripe-combinatorial}
  Let $\lambda$ and $k$ be positive integers.
  Then there exists a constant $K_3$ depending only on $X$ such that
  \begin{equation*}
    \bigg| \stripe_\lambda^{(m)}(A)^* \isect
    \Big(\Union_{B\in \scr E^{(m)}(A)} \stripe_\lambda^{(m)}(B)^* \union \stripe_\lambda^{(n)}(B)^* \Big)
    \bigg|
    \leq K_3\, q^{k\cdot \varepsilon}\cdot \big| \stripe_\lambda^{(m)}(A)^* \big|
  \end{equation*}
  for all $1\leq m,n\leq M$ and $A\in\dcubes$, where
\[
\scr E^{(m)}(A):= \bigcup \,\{\scr E_{d\cdot k}^{(m)}(A)\, :\, d\in\mathbb{N},\ 1\leq d\cdot k\leq \lambda-1\}.
\]
\end{lem}

\begin{proof}
First, observe that
\begin{align*}
  \bigg| \stripe_\lambda^{(m)}(A)^* \isect \Big(\Union_{B\in \scr E^{(m)}(A)}
  &\stripe_\lambda^{(m)}(B)^* \union \stripe_\lambda^{(n)}(B)^* \Big) \bigg|\\
  &\leq \sum_{B \in\scr
    E^{(m)}(A)} |\stripe_\lambda^{(m)}(B)^*|+ |\stripe_\lambda^{(n)}(B)^*|.
\end{align*}
Now we use (S\ref{enu:S2}) to dominate this expression by
\begin{equation}\label{eq:stripe-lemma-1}
(1+K_1)\sum_{B \in\scr E^{(m)}(A)} |\stripe_\lambda^{(m)}(B)^*|.
\end{equation}
Note that (S\ref{enu:S1}) and (S\ref{enu:S2}) also give us
\[
|\stripe_\lambda^{(m)}(B)^*| \leq \frac{K_1}{M} |B|.
\]
The latter inequality implies that~\eqref{eq:stripe-lemma-1} is bounded from above by
\begin{equation}\label{eq:stripe-lemma-2}
\frac{(1+K_1)K_1}{M}\, \Bigg( \sum_{d\,:\, 1\leq d\cdot k < \lambda} \sum_{B\in \scr E_{d\cdot k}^{(m)}(A)}|B| \Bigg).
\end{equation}
Employing (S\ref{enu:S4}) we estimate~\eqref{eq:stripe-lemma-2} by
\[
\frac{(1+K_1)K_1 \,K_2}{M}\, \Bigg(
  \sum_{d\, :\, 1\leq d\cdot k < \lambda} q^{d\cdot k\cdot\varepsilon}|A|
\Bigg).
\]
Finally, applying (S\ref{enu:S2}) concludes the proof of the lemma.
\end{proof}

\begin{thm}\label{thm:stripe}
  Let $X$ be a space of homogeneous type, $E$ a $\umd$--space and $1 < p < \infty$.
  Let $\lambda$ be a positive integer, then there exists a constant $C$ such that
  \begin{equation*}
    \big\| S_\lambda^{(m)} f \big\|_{L_E^p(X)}
    \leq C\, \big\| S_\lambda^{(n)} f \big\|_{L_E^p(X)},
    \qquad \text{$f \in \overline{\lin}\{h_Q:Q\in\scr Q\}$},
  \end{equation*}
  for all $1 \leq m,n \leq M$.
  The constant $C$ depends only on $p$, $X$ and $E$.
\end{thm}

\begin{proof}
  Let $\lambda \geq 1$ and $m\neq n$ be fixed throughout the proof.
  Define $k$ as the smallest positive integer such that
  $K_3\, q^{k\cdot \varepsilon} \leq \frac{1}{4C_g^2}$, where $K_3$, $\varepsilon$ and $C_g$ are the
  constants appearing in Lemma~\ref{lem:stripe-combinatorial}, (S\ref{enu:S4}) and~(G\ref{enu:G3}),
  respectively.
  Moreover, define the collections
  \begin{align*}
    \scr C_{j,\nu}^{(\delta)}
    & := \Union_{\substack{0 \leq i \leq \lambda-1\\i\mod k = \nu}}
      \dcubes_{(2 j + \delta) \lambda + i},
      & & \text{$j \in \bb Z$, $\delta \in\{0,1\}$, $0 \leq \nu \leq k-1$},\\
    \scr C_\nu^{(\delta)}
    & := \Union_{j \in \bb Z} \scr C_{j,\nu}^{(\delta)},
    & & \text{$\delta \in\{0,1\}$, $0 \leq \nu \leq k-1$}.
  \end{align*}
  With $\nu$ and $\delta$ fixed, set
  \begin{equation*}
    A(Q)
    := \big( \scr S_\lambda^{(m)}(Q)^* \union \scr S_\lambda^{(n)}(Q)^* \big)
    \setminus \Union_{\substack{P \in \scr C_{j,\nu}^{(\delta)}\\\lev P > \lev Q}} A(P),
    \qquad\text{$Q \in \scr C_{j,\nu}^{(\delta)}$}
  \end{equation*}
  for each $j \in \bb Z$.
  This definition is understood as an induction on $\lev Q$, starting with the
  maximal level in $\scr C_{j,\nu}^{(\delta)}$.
  Note that the above union is empty if $\lev Q$ is maximal in
  $\scr C_{j,\nu}^{(\delta)}$.
  Now, Lemma~\ref{lem:stripe-combinatorial} and our choice of $k$ imply
  \begin{equation}\label{proof:thm:stripe-4}
    \big| A(Q) \isect \stripe_\lambda^{(n)}(Q)^* \big|
    \geq \Big( 1 - \frac{1}{4 C_g^2} \Big)\, |\stripe_\lambda^{(n)}(Q)^*|.
  \end{equation}
  We collect all the sets $A(Q)$ in $\scr A$, to be more precise
  \begin{equation*}
    \scr A
    := \big\{ A(Q)\, :\, Q \in \scr C_\nu^{(\delta)}\big\}.
  \end{equation*}
  The inductive construction of $A(Q)$ is performed in such a way that $\scr A$ is nested, which we
  shall show in the following.
  Observe, if $P,Q \in \scr C_{j,\nu}^{(\delta)}$, then
  $A(P) \isect A(Q) = \emptyset$.
  Moreover, if $Q \in \scr C_{j,\nu}^{(\delta)}$, then $A(Q)$ comprises of cubes in
  $\dcubes_{\lev Q + 2\lambda-1}$.
  Thus, if $P \in \scr C_{i,\nu}^{(\delta)}$ and
  $Q \in \scr C_{j,\nu}^{(\delta)}$ with $i < j$, then
  $A(Q) \subset Q \subset A(P)$, provided $A(P) \isect A(Q) \neq \emptyset$.
  Hence, $\scr A$ is a nested collection.
  Let us define
  \begin{equation*}
    \scr A_j := \big\{ A(Q)\in \scr A\, :\, Q \in \dcubes_j \big\},
    \qquad j \in \bb Z,
  \end{equation*}
  and the filtrations $\{\cal F_j\}$ and $\{\cal G_j\}$ by
  \begin{align*}
    \cal F_j &:= \salg \big( \Union_{i \leq j} \scr A_i \big),
    & \cal G_j &:= \salg \big(\big\{ \stripe_\lambda^{(m)}(Q)^*\, :\, Q \in
      \dcubes_i,\ i\leq j
    \big\}\big)
  \end{align*}
  for all $j \in \bb Z$.
  Note that some of the sets $\scr A_j$ are empty.
  So, if $\scr A_j = \emptyset$, we delete the $\sigma$--algebras $\cal F_j$ and $\cal G_j$ from
  their respective filtrations.

  Let $f \in L_E^p(X)$ have the representation $f = \sum_{Q \in \scr C_\nu^{(\delta)}} f_Q\, h_Q$.
  Due to (G\ref{enu:G2}), the $\umd$--property and Kahane's contraction principle yield
  \begin{equation}\label{proof:thm:stripe-1}
    \big\| S_\lambda^{(m)} f \big\|
    = \Big\| \sum_{Q \in \scr C_\nu^{(\delta)}} f_Q\, g_{Q,\lambda}^{(m)} \Big\|
    \lesssim \int_0^1 \Big\|
      \sum_{Q \in \scr C_\nu^{(\delta)}} r_Q(t)\, f_Q\, \big| g_{Q,\lambda}^{(m)}\big|
    \Big\|
    \dif t.
  \end{equation}
  First, observe that
  \begin{equation*}
    \charfun_{\stripe_\lambda^{(m)}(Q)^*}
    \leq \frac{|\stripe_\lambda^{(m)}(Q)^*|}{|A(Q)|}\,
      \cond(\charfun_{A(Q)} \mid \cal G_{\lev Q}),
      \qquad Q \in \scr C_\nu^{(\delta)}.
  \end{equation*}
  Secondly, due to our choice of $k$, we obtain from Lemma~\ref{lem:stripe-combinatorial} that
  $|\stripe_\lambda^{(m)}(Q)^*| \leq \frac{4}{3}\, |A(Q)|$ for all $Q \in \scr C_\nu^{(\delta)}$.
  The latter two inequalities imply
  \begin{equation}\label{proof:thm:stripe-2}
    \big| g_{Q,\lambda}^{(m)} \big|
    \leq \big\| g_{Q,\lambda}^{(m)} \big\|_\infty
      \cdot \charfun_{\stripe_\lambda^{(m)}(Q)^*}
    \leq \frac{4}{3}\, \big\| g_{Q,\lambda}^{(m)} \big\|_\infty
      \cdot \cond(\charfun_{A(Q)} \mid \cal G_{\lev Q}),
    \qquad Q \in \scr C_\nu^{(\delta)}.
  \end{equation}
  Combining~\eqref{proof:thm:stripe-1} and~\eqref{proof:thm:stripe-2}, together with Kahane's
  contraction principle yield
  \begin{equation*}
    \big\| S_\lambda^{(m)} f \big\|
    \lesssim \int_0^1 \Big\|
      \sum_{Q \in \scr C_\nu^{(\delta)}} r_Q(t)\, f_Q\,
      \big\| g_{Q,\lambda}^{(m)}\big\|_\infty\,
      \cond(\charfun_{A(Q)} \mid \cal G_{\lev Q})
    \Big\|
    \dif t.
  \end{equation*}
  Applying Bourgain's version of Stein's martingale gives
  \begin{equation}\label{proof:thm:stripe-3}
    \big\| S_\lambda^{(m)} f \big\|
    \lesssim \int_0^1 \Big\|
      \sum_{Q \in \scr C_\nu^{(\delta)}} r_Q(t)\, f_Q\,
      \big\| g_{Q,\lambda}^{(m)}\big\|_\infty\, \charfun_{A(Q)}
    \Big\|
    \dif t.
  \end{equation}
  By (G\ref{enu:G1}), the support of $g_{Q,\lambda}^{(n)}$ is a subset of
  $\stripe_\lambda^{(n)}(Q)^*$.
  If we define
  \begin{equation*}
    V:=\Big\{
      |g_{Q,\lambda}^{(n)}| \geq \frac{\| g_{Q,\lambda}^{(m)} \|_\infty}{2C_g}
    \Big\} \isect A(Q)\isect \stripe_\lambda^{(n)}(Q)^*,
  \end{equation*}
  then~\eqref{proof:thm:stripe-4} and Lemma~\ref{lem:stripe-functions} imply
  \begin{equation}\label{proof:thm:stripe-5}
    | V |
    \geq \frac{1}{4C_g^2}\, \big| \stripe_\lambda^{(n)}(Q)^* \big|
    \geq \frac{1}{4C_g^2\,(1+K_1)}\, | A(Q) |.
  \end{equation}
  As a consequence of the definition of $V$ and~\eqref{proof:thm:stripe-5},
  \begin{align*}
    \big\| g_{Q,\lambda}^{(m)}\big\|_\infty\cdot \charfun_{A(Q)}
    & \leq \Big( \frac{2C_g}{|V|} \int_V |g_{Q,\lambda}^{(n)}|\,
    \Big)\cdot \charfun_{A(Q)}\\
    & \leq \Big( \frac{8C_g^3(1+K_1)}{|A(Q)|}
      \int\limits_{A(Q)} |g_{Q,\lambda}^{(n)}|\, \Big)
      \cdot \charfun_{A(Q)}\\
    & \leq 8C_g^3(1+K_1)\cdot \cond(|g_{Q,\lambda}^{(n)}| \mid \cal F_{\lev Q})
  \end{align*}
  for all $Q \in C_\nu^{(\delta)}$.
  Plugging the latter estimate into~\eqref{proof:thm:stripe-3}, Kahane's contraction principle
  yields
  \begin{equation*}
    \big\| S_\lambda^{(m)} f \big\|
    \lesssim \int_0^1 \Big\|
      \sum_{Q \in \scr C_\nu^{(\delta)}} r_Q(t)\, f_Q\,
      \cond(|g_{Q,\lambda}^{(n)}| \mid \cal F_{\lev Q})
    \Big\|
    \dif t.
  \end{equation*}
  Subsequently, applying Stein's martingale inequality, Kahane's contraction
  principle to pass from $|g_{Q,\lambda}^{(n)}|$ to $g_{Q,\lambda}^{(n)}$ and
  finally using the 
  $\umd$--property to dispose of the Rademacher functions, concludes the proof.
\end{proof}

Applying the estimate in Theorem~\ref{thm:stripe}, i.e., the uniform equivalence of stripe operators,
we obtain upper and lower estimates for $S_\lambda^{(m)}$ via the cotype and type inequalities,
respectively.
\begin{thm}\label{thm:stripe-main}
  Let $X$ be a space of homogeneous type, $E$ a $\umd$--space and $1<p<\infty$.
  Moreover, let $\lambda$ be a positive integer and $1\leq m\leq M$. If we assume
  \begin{equation}\label{eq:stripe:assumption}
    \big\|\sum_{n=1}^M S_\lambda^{(n)}h_Q\big\|_\infty\leq  C_S \cdot \frac{1}{|Q|}\int |h_Q|,
    \qquad Q\in\dcubes,
  \end{equation}
  then
  \begin{equation}\label{eq:stripe:assertion}
    \|S_\lambda^{(m)}f\|_{L_E^p(X)} \leq C\cdot C_S M^{-1/\cal C} \|f\|_{L_E^p(X)},\qquad f\in \overline\lin\{h_P:P\in \scr Q\},
  \end{equation}
  where $L_E^p(X)$ has cotype $\cal C$ and the constant $C$ depends only on  $p$, $X$ and $E$.

  On the other hand, if we assume
  \begin{equation}\label{eq:stripe:assumption:2}
    \|h_Q\|_\infty\leq C_S\cdot \sum_{n=1}^M \frac{1}{|Q|}\int |S_\lambda^{(n)} h_Q|,\qquad Q\in\dcubes,
  \end{equation}
  then
  \begin{equation}\label{eq:stripe:assertion:2}
    \|S_\lambda^{(m)}f\|_{L_E^p(X)} \geq C\cdot C_S^{-1} M^{-1/\cal T} \|f\|_{L_E^p(X)},\qquad f\in \overline\lin\{h_P:P\in \scr Q\},
  \end{equation}
  where $L_E^p(X)$ has type $\cal T$ and the constant $C$ depends only on  $p$, $X$ and $E$.
\end{thm}
\begin{proof}
  First, we prove inequality~\eqref{eq:stripe:assertion} under the
  hypothesis~\eqref{eq:stripe:assumption}.
  Let $f=\sum_Q f_Q h_Q$ be a finite sum and $m$ be an integer in the range $1\leq m\leq M$.
  By (M\ref{enu:M2}), $\{h_Q\}$ is a martingale difference sequence, thus
  \[
  \|f\|\gtrsim \int_0^1 \Big\| \sum_Q r_Q(t) f_Q |h_Q|\Big\| \dif t
  \] 
  as a consequence of the $\umd$--property of $E$ and Kahane's contraction principle.
  Define the filtration $\{\cal F_j\}$ by
  \[
  \cal F_j = \salg \big( \Union_{i \leq j} \dcubes_i \big).
  \]
  Then, Bourgain's version of Stein's martingale inequality yields
  \begin{equation}\label{eq:stripe:bound:1}
    \|f\|\gtrsim \int_0^1 \Big\| \sum_Q r_Q(t) f_Q \cond(|h_Q|\mid \cal F_{\lev Q})\Big\| \dif t.
  \end{equation}
  Observe that $Q$ is an atom in the $\sigma$-algebra $\cal F_{\lev Q}$ for all $Q\in\dcubes$, and thus
  \[
  \cond(|h_Q|\mid \cal F_{\lev Q}) = \Big(\frac{1}{|Q|}\int |h_Q| \Big)\cdot\charfun_{Q}
  \geq C_S^{-1}\cdot \Big\|\sum_{n=1}^M S_\lambda^{(n)} h_Q\Big\|_\infty\charfun_Q,
  \]
  where we used (M\ref{enu:M1}) and \eqref{eq:stripe:assumption}.
  Plugging the latter inequality into \eqref{eq:stripe:bound:1}, Kahane's contraction principle
  implies
  \begin{equation*}
    \|f\|\gtrsim C_S^{-1} \int_0^1 \Big\| \sum_Q r_Q(t) f_Q\sum_{n=1}^M S_\lambda^{(n)}h_Q \Big\| \dif t.
  \end{equation*}
  Condition~(G\ref{enu:G2}) and the $\umd$--property of $L_E^p(X)$ yield
  \begin{equation}\label{eq:stripe:bound:2}
    \|f\| \gtrsim C_S^{-1} \Big\|
      \sum_{j\in\mathbb{Z}}\sum_{n=1}^M S_\lambda^{(n)}\Big(\sum_{Q\in\dcubes_j} f_Q h_Q\Big)
    \Big\|.
  \end{equation}
  Now let
  \[
  d_{j,n} := S_\lambda^{(n)}\Big(\sum_{Q\in\dcubes_j} f_Q h_Q\Big)
  \]
  and observe that $(d_{j,n})$ constitutes a martingale difference sequence with respect to the
  lexicographic ordering on the index pairs $(j,n)$.
  Thus, \eqref{eq:stripe:bound:2} implies
  \begin{equation*}
    \|f\|\gtrsim C_S^{-1} \int_0^1\Big\| \sum_{n=1}^M r_n(t)\sum_{j\in\mathbb{Z}} d_{j,n} \Big\|\dif t.
  \end{equation*}
  Since $L_E^p(X)$ has cotype $\cal C$, we employ the cotype inequality to obtain
  \begin{align*}
    \|f\|
    \gtrsim C_S^{-1} \Big(
      \sum_{n=1}^M \big\| \sum_{j\in\mathbb{Z}} d_{j,n} \big\|^{\cal C}
    \Big)^{1/\cal C}
    = C_S^{-1} \Big(
      \sum_{n=1}^M \big\|  S_\lambda^{(n)}f \big\|^{\cal C}
    \Big)^{1/\cal C}.
  \end{align*}
  By Theorem \ref{thm:stripe}, $\|S_\lambda^{(n)}f \| \gtrsim \|S_\lambda^{(m)}f \|$ for all $1\leq n\leq M$, and therefore
  \[
  \|f\|\gtrsim C_S^{-1} M^{1/\cal C}\|S_\lambda^{(m)}f \|,
  \]
  proving \eqref{eq:stripe:assertion}. 

  A similar argument replacing the cotype inequality by the type inequality proves
  \eqref{eq:stripe:assertion:2} under the condition \eqref{eq:stripe:assumption:2}.
\end{proof}

\subsection*{Acknowledgments}
The authors are grateful to P. F. X. M\"uller for arranging this collaboration and many enlightening
discussions.
We would like to thank the referee whose suggestions led to the elimination of two technical
assumptions originally imposed on the space of homogeneous type.
M.\,P. is supported by the Austrian Science Fund, FWF Project P 23987-N18.

\bibliographystyle{plain}
\bibliography{bibliography}

\end{document}